\documentclass[11pt]{article}
\author{Lennart Ronge}
\title{Formulas for Hadamard coefficients in terms of Green's operators}
\usepackage{amsmath}
\usepackage{amssymb}
\usepackage{amsfonts}
\usepackage[ansinew]{inputenc}
\usepackage{amsthm}
\usepackage{geometry}
\usepackage{hyperref}
\usepackage[UKenglish]{babel}
\usepackage{bbm}
\usepackage{tikz}
\usepackage{graphicx}
\usepackage {wrapfig}
\usepackage {csquotes}
\usepackage[style=alphabetic, backend=bibtex]{biblatex}
\def\<{\left\langle}
\def\>{\right\rangle}

\def\C{{\mathbb{C}}}
\def\R{{\mathbb{R}}}
\def\N{{\mathbb{N}}}

\def\insum#1{\sum\limits_{#1=0}^\infty}
\def\intR{\int\limits_\R}
\def\intP{\int\limits_0^\infty}

\def\O {\backslash \{0\}}
\def\F {\mathcal{F}}

\def\M{\mathcal{M}}
\def\casedist#1#2#3#4{
	\begin{cases}
		#1& \text{, if } #2\\
		#3& \text{, if } #4 \\
	\end{cases}}

\newtheorem{df}{Definition}[section]
\newtheorem{lemma}[df]{Lemma}
\newtheorem{thm}[df]{Theorem}
\newtheorem*{thm*}{Theorem}
\newtheorem{prop}[df]{Proposition}
\newtheorem{cor}[df]{Corollary}

\newtheorem{gn}[df]{Fixed Notation}
\newtheorem{rem}[df]{Remark}
\newtheorem{defprop}[df]{Definition/Proposition}

\DeclareMathOperator{\supp}{supp}

\DeclareMathOperator{\grad}{grad}
\DeclareMathOperator{\sign}{sign}

\DeclareMathOperator{\scal}{scal}

\begin{document}
\maketitle
\begin{abstract}
 We describe various ways of obtaining the Hadamard coefficients associated to a normally hyperbolic operator from the corresponding Green's operators. As the Hadamard expansion on its own is not enough for this, we include additional information either by considering something like a resolvent or powers of Green's operators or by looking at a product of the original manifold with the real line.
\end{abstract}
\section{introduction}
Hadamard coefficients are functions in two variables associated to a normally hyperbolic operator $P$ that can be viewed as the Lorentzian equivalent of the Riemannian heat kernel coefficients. They play an important role, for example, in the local equivariant index theorem (\cite{BS}) or in the Hadamard renormalization of the stress-energy tensor (\cite{DeFo}). The Hadamard coefficients, particularly their values on the diagonal, also encode information about the underlying manifold and the associated wave operator. The first coefficient (starting count at $0$) for the d'Alembertian is proportional to the scalar curvature of the manifold (or from a physics point of view, the density of the Einstein-Hilbert action). The coefficient at half the dimension for the square of a Dirac operator is related to the index density for this operator.

Our goal here is to develop a formula for expressing the Hadamard coefficients in terms of the Green's operators of a given wave operator (or suitable generalizations thereof). A particular motivation for this comes from the spectral action (see e.g. \cite{CoCh}) which has an asymptotic expansion in terms of heat kernel coefficients (the Riemannian counterpart of the Hadamard coefficients). Having a Lorentzian version of a spectral action would thus give one the Hadamard coefficients, while an operator theoretic description of the Hadamard coefficients might lead the way towards a spectral action in Lorentzian settings where the wave operator is not self-adjoint (for a version in certain settings where the wave operator is self-adjoint, see \cite{DaWr}).

The purpose of this paper is twofold: to present the main results of \cite{Ron} in a more accessible fashion (the thesis was written for completeness and rigor, sometimes at the cost of readability), which will be done up to Section \ref{Hadamextract}, and to present a new, alternative approach using a product construction that yields different, simpler results, which will be developed in sections \ref{ProdDiag} and \ref{ProdOffDiag}.

A brief overview of the main objects and definitions is given at the end of this introduction (\ref{overview}) and can be used for easy reference. 

The main motivation for defining the Hadamard coefficients comes from the Hadamard expansion for the (advanced/retarded) Green's operators (something like inverses to $P$). This reads
\[G^\pm\sim \insum{k}V^kR^\pm (2k+2,\cdot),\]
where the $V^k$ are the Hadamard coefficients and $R^\pm$ are the (advanced/retarded) Riesz distributions (morally, something like powers of the distance function). Ideally, we would like to "invert" this expansion to obtain a formula for the Hadamard coefficients (at least on the diagonal) in terms of $G^\pm$. However, this is hopeless, for much the same reason that a sum
\[\sum\limits_{k=0}^Na_k(x) x^k\]
does not determine the coefficients $a_k$ (not even at $x=0$). We thus need to somehow introduce an extra parameter to get access to additional information and make the problem well-posed.

The starting point is to integrate the kernel $G_x$ of the difference $G=G^+-G^-$ at $x$ along a timelike curve $w$ through $x$, paired with a test function $f$. Rescaling $f$ by a parameter $s$, we obtain a function (informally written)
\begin{equation}
	\label{dfl}
	L(s):=\intR G_x(w(t))f(\tfrac{t}{s})dt
\end{equation}
whose asymptotic behavior captures the asymptotic behavior of $G_x$ near $x$ (along the chosen geodesic) and is thus related to the Hadamard coefficients. In section \ref{caf}, we show (Theorem \ref{intexp}) that this has an asymptotic expansion of the form (for $w$ a unit speed geodesic)
\[L(s)\sim\sum\limits_{k,n=0}^\infty a(k,n)(V_x^k\circ w)^{(2n)}(0)s^{2k+2n+3-d}.\]
For product spacetimes, the time derivatives of the Hadamard coefficients and hence all summands for $n\neq 0$ vanish and we can read off the Hadamard coefficients directly from the expansion of $L(s)$ (this is no accident - for product spacetimes $L(s)$ is directly related to the wave trace and the spectral action on the corresponding space). However, for general spacetimes we have additional error terms that make it impossible to extract the Hadamard coefficients from this alone.

The first way to circumvent this problem is to look at something like a resolvent, i.e. to introduce an auxiliary parameter $z\in \C$ and consider the Green's operators for $P-z$ instead of $P$. This is the approach we will pursue in Section \ref{Hadamextract} (equivalently, one could also consider powers of Green's operators, see Corollary \ref{Vwithpowers}). The Hadamard coefficients of $P-z$ can be expressed in terms of the Hadamard coefficients of $P$ as (Proposition \ref{Vz})
\[V^k(z)=\sum\limits_{m=0}^k\binom{k}{m}z^mV^{k-m}.\]
Inserting this into the expansion of $L$ for $P-z$ gives an additional $z$-dependence. This allows us to eliminate the derivative terms and solve for the Hadamard coefficients on the diagonal. We obtain our first main theorem.
\begin{thm*}[\ref{main}]
	Assume the setting laid out in \ref{overview}. We have for any $o\in \N_0$
	\[V^k_x(x)=\sum\limits_{m=0}^k q(k,m,o)L_{k,m+o},\]
	where 
	\[q(k,m,o):=\frac{\pi^{\frac{d}{2}-1}4^{k+m+o}(m+o)!k!}{\M'(f)({2k+2m+2o-d+3})}\binom{\frac{d}{2}-1-o-k}{m}\binom{2k+o+1-\frac{d}{2}}{k-m}\]
	are some coefficients independent of the geometry of $M$ and
	\[L_{k,m+o}=w^*(G_{P-z,x})[f(\tfrac{\cdot}{s})][[s^{2k+2m+2o-d+3}]][[z^{m+o}]]\]
	are the expansion coefficients of the function $L$ in (\ref{dfl}), defined for $P-z$ instead of $P$, that correspond to the given powers of $s$ and $z$.
\end{thm*}
While this formula is somewhat complicated, it can be simplified by suitable choices of $o$, considerably so for the Hadamard coefficients with $k<\tfrac{d}{2}$ in the even dimensional case (see Remark \ref{qsipm}). 

The other way to introduce additional information is to consider Green's operators on a product $\tilde M:= M\times \R$ (with spacelike $\R$). The crucial fact we use here is that the Hadamard coefficients are independent of the additional product parameter, while the Riesz distributions are not. In Section \ref{ProdDiag}, we will use this to find another formula for the diagonal Hadamard coefficients. We replace the curve $w$ in the above construction by 
\[w_\xi(x,t):=(w(\xi(t)),\sqrt{\xi^2-1}\ t),\]
which is unit speed in $\tilde M$, while the speed of its projection to $M$ is proportional to $\xi$. Thus the derivatives of the Hadamard coefficients along $\xi$ will be proportional to powers of $\xi$, while the Riesz distributions will, to leading order, be unaffected. As a consequence,  taking the expansion coefficient for $\xi^0$ will eliminate all the error terms involving derivatives of the Hadamard coefficients. We obtain
\begin{thm*}[\ref{prodmain}]
Assume the setting laid out in \ref{overview} for $\tilde M=M\times\R$. With $w_\xi$ as defined above, we have
\[V^k_x(x)=\frac{4^k k!}{\pi^{\frac{1-d}{2}}\M'(f)(2k+2-d)}w_\xi^*(\tilde G)[f_s][[s^{2k+2-d}]][[\xi^0]].\]
\end{thm*}
Unfortunately, the construction requires $\xi >1$, so we cannot simply set $\xi=0$ here. The geometric reason is that a curve that is constant in $M$-direction would be spacelike and thus immediately leave the support of $G_x$.
This does, however work if we want to compute coefficients away from the diagonal, for $y$ in the future or past of $x$. In that case we can simply integrate in the product direction, using a cutoff to localize around the edge of the causal future of $x$. As the Hadamard coefficients are constant on these lines and the Riesz distributions have singularities/zeros of different orders at the edge, we can directly extract the Hadamard coefficients from the expansion at the boundary.
This leads to
\begin{thm*}[Corollary \ref{offdV}]
	Assume again the setting described in \ref{overview}
	Then for $y\in I^\pm(x)$, $k\in \N$ and
	\[\chi_{\epsilon,x,y}(r):=|r|\chi\left(\frac{\Gamma_x(y)-r^2}{\epsilon}\right),\] we have
	\[V^k_x(y)=2^{2k+1}k!\pi^{\tfrac{d-1}{2}}\frac{\Gamma}{\M(\chi)}(k+\tfrac{3-d}{2})\iota_y^*(\tilde G^\pm_x)[\chi_{\epsilon,x,y}][[\epsilon^{k+\tfrac{3-d}{2}}]].\]
\end{thm*}
\subsection{Overview of the main definitions}
\label{overview}
	We provide here a brief overview over the objects that will be defined and frequently used later on.
	\begin{itemize}
		\item $M$ is a globally hyperbolic Lorentzian manifold of dimension $d\geq 2$.
		\item $U\subset M$ is open, globally hyperbolic, causally compatible and geodesically convex.
		\item $P$ is a normally hyperbolic operator over $M$ (\ref{dfP}).
		\item $x$ is an arbitrary point in $U$ that serves as basepoint.
		\item $G^\pm_{P,x}$ denotes the distributional kernels of the advanced/retarded Green's operators of $P$, with $x$ inserted into the second argument. We define $G_{P,x}:=G^+_{P,x}-G^-_{P,x}$.
		Unless relevant, we usually omit $P$ from the notation (\ref{dfG}).
		\item $V^k$ are the Hadamard coefficients for $P$ on $U$ (\ref{dfV}).
		\item $\Gamma$ denotes the "Lorentzian geodesic distance squared" (\ref{dfGamma}) or, depending on the context and the argument, the gamma function.
		\item $R^\pm$ are the advanced/retarded Riesz distributions and $R:=R^+-R^-$ (\ref{dfRiesz}).
		\item $w$ is a timelike curve with $w(0)=x$ and unit initial speed, sometimes (in Section \ref{Hadamextract} and the results thereof) assumed to be a unit speed geodesic.
		\item $\M$ denotes the Mellin transform and $\M':=\frac{\M}{\Gamma(\tfrac{\cdot+1}{2})}$.
		\item $f\in C_c^\infty(\R)$ is an odd function whose Mellin transform has non-vanishing values or residues whenever we need to divide by them.
		\item $[\cdot]$ denotes evaluation of a distribution at a test function
		\item $[[\cdot]]$ denotes taking expansion coefficients, i.e. if $a(s)\sim \insum{k}a_ks^k$, then $a[[s^n]]:=a_n$.
		\item $\iota_x(t):=(x,t)$
		\item Objects with a tilde are the corresponding objects on $\tilde M:=M\times \R$ (\ref{dfMtilde}, \ref{dftilde}).
	\end{itemize}

\section{Preliminaries}
\subsection{The building blocks}
In this section, we introduce Green's operators, Riesz distributions and Hadamard coefficients, which will play a central role in this thesis. We will follow \cite{BGP} here, other descriptions can be found in in \cite {Gun} and \cite {Fri}.

Some of the objects defined in this section can only be constructed over a geodesically convex set as they depend on choices of geodesics. To make sure that Green's operators will restrict well to this subset, we need to require further properties.

\begin{gn}
	Throughout this paper, let $M$ be a globally hyperbolic Lorentzian manifold of dimension $d\geq 2$ and fix a globally hyperbolic, causally compatible, geodesically convex open subset $U$ of $M$.
\end{gn}
Globally hyperbolic means the manifold has a cauchy hypersurface, i.e. a surface intersected exactly once by each inextendable causal curve. Convexity means that any two points can be joined by a unique geodesic. Causal compatibility means that two points in $U$ that are causally related in $M$ are also causally related in $U$.
As a consequence of \cite[Corollary 2 and Remark 14]{Min}, every point in $M$ has a Basis of neighborhoods that have these properties.

Normally hyperbolic operators are the Lorentzian equivalent of  Laplace-type operators. They carry information about the geometry of the manifold. Green's operators are something like inverses for these. The standard example for a normally hyperbolic operator on $M$ is the d'Alembertian.
\begin{df}
	\label{dfP}
	A normally hyperbolic operator on $M$ is a second order differential operator on a vector bundle $E$ over $M$ whose principal symbol is given by minus the metric $g$ (times the identity on fibres).
\end{df}
\begin{gn}
	For the rest of this paper, let $P$ be a normally hyperbolic operator on a vector bundle $E$ over $M$. \label{gnP}
	For the sake of simplicity and clarity, we will assume that $E|_{U}=U\times\C^k$ so that we can compare different fibres without further complications. This constitutes no loss of generality, as any bundle over the contractible set $U$ is trivializable.
\end{gn}
Every normally hyperbolic operator can be written as $\square^\nabla+B$ for some connection $\nabla$ and some bundle endomorphism $B$. These are uniquely determined by the Operator. In the following, we denote by $\nabla$ the connection induced by $P$ in this way. With this, we have the Leibnitz rule
\[P(fh)=fPg-2\nabla_{\grad(f)}h+(\square f)h.\]

On the usual spaces of functions, $P$ will not be invertible. Instead, we have a well-posed Cauchy problem (see \cite[Theorem 3.2.11]{BGP}). To get something invertible, one needs to specify support conditions (see \cite[Definition 3.4.1 and Corollary 3.4.3]{BGP} and \cite[Theorem 3.8 and Lemma 4.1]{GH}):
\begin{defprop}
	\label{dfG}
	A set is called past/future compact, if its intersection with the past/future of every point is compact. Let $\Gamma_\pm(E)$ denote the set of smooth sections of $E$ whose support is past/future compact.
	
	The Operator $P$ viewed as a map $\Gamma_\pm(E)\rightarrow \Gamma_\pm(E)$ is invertible. Its inverse is called the advanced/retarded Green's operator and denoted by $G^\pm$. These have the property that $G^\pm(f)$ is supported in the past/future of $\supp(f)$. $G^\pm$ continuously extend to distributions with past/future compact support.
	
	More generally, for any normally hyperbolic operator $Q$, we denote by $G^\pm_Q$ the advanced/retarded Green's operators associated to $Q$.
	
	We denote by $G^\pm_x$ the schwartz kernel of $G^\pm$ evaluated at $x$ in the second coordinate (informally $G^\pm(\cdot,x)$, more formally $G^\pm(\delta_x)$).
\end{defprop}
Note that the definition of $G^\pm$ on these slightly unusual function spaces not only allows us to view the Green's operators as actual inverses of $P$, it also allows us to take powers of Green's operators without problems.

As $U$ is globally hyperbolic, Green's operators also uniquely exist on $U$ instead of $M$. As $U$ is causally compatible, by \cite[Proposition 3.5.1.]{BGP} we have the following:
\begin{prop}
	For a test section $\phi$ supported in $U$, we have
	\[G^\pm_{P|_U}(\phi|_U)=(G^\pm_{P}\phi)|_U.\]
\end{prop}
This means that the Schwartz kernel of $G^\pm_{P|_U}$ is just the restriction of that of $G^\pm_{P}$. Justified by this compatibility, we will not distinguish between these two in the following.

We now turn to the Riesz distributions. As all constructions rely heavily on the exponential map, the Riesz distributions (and later on also the Hadamard coefficients) will only be defined on convex subsets of $M$. Note that their values will be dependent on the choice of convex subset.

The Riesz distributions on $U$ are roughly speaking a Lorentzian analogue to "powers of the distance functions". They are described in detail in \cite[sections 1.2 and 1.4]{BGP}. Here, they  will play a role comparable to that of powers of $x$ in a Taylor series. Before we can define Riesz distributions on subsets of $M$, we will first define them on Lorentzian vector spaces (up to isomorphism, this just means Minkowski space, but the primary use case is on tangent spaces of $M$), where they are, roughly speaking, "powers of the norm".
\begin{df}
	Let $V$ be a Lorentzian vector space with Lorentzian bilinear form $\eta$.
	For $x\in V$ define 
	\[\gamma(x):=-\eta(x,x).\]
\end{df}
Note that $\gamma(x)$ strictly positive/negative if and only if $x$ is timelike/spacelike.
\begin{defprop}
	\label{dfRiesz}
	Let $V$ be a d-dimensional Lorentzian vector space.
	For $\alpha\in \C$ with $\Re(\alpha)\geq 0$, define Riesz distributions $R_\pm(\alpha)$ as the distribution on $V$ given by the function
	\[R_\pm(\alpha)(x)=\casedist{c_\alpha\gamma(x)^{\frac{\alpha-d}{2}}}{x\in J_\pm(0)}{0}{x\notin J_\pm(0)}\]
	with
	\[c_\alpha:=\frac{2^{1-\alpha}\pi^\frac{2-d}{2}}{\Gamma(\frac{\alpha}{2})\Gamma(\frac{\alpha-d+2}{2})}.\]
	The map $\alpha\mapsto R_\pm(\alpha)$ is holomorphic as distribution-valued map and extends uniquely to a holomorphic map on all of $\C$. For arbitrary $\alpha\in\C$, define $R_\pm(\alpha)$ to be the value of this holomorphic extension. (The prefactor $c_\alpha$ is chosen such that $\square R_\pm(\alpha+2)=R_\pm(\alpha)$.) In case it is clear what $V$ is, we omit it from the notation.
\end{defprop}
We transport this from tangent spaces to our convex subset $U$ using the exponential map:
\begin{df}
	\label{dfGamma}
	For $x,y\in U$ and $\alpha\in \C$,
	define the "squared geodesic distance"
	\[\Gamma_x(y):=\gamma\left((exp_x)^{-1}(y)\right)\]
	and define the Riesz distributions
	\[R_\pm(\alpha,x):=(exp_x)^{-1*}(R_\pm(\alpha)|_{Dom(exp_x)})\]
	on $U$ (here $exp_x$ refers to the exponential map on $U$, which is a diffeomorphism, as $U$ is convex). 
\end{df}
Note that $\Gamma$ and $R_\pm(\alpha,x)$ depend on the choice of $U$. However, as we do not vary $U$, we omit this dependence from the notation.
$\Gamma_x$ is positive precisely on $J(x)$ and $R_\pm(\alpha,x)$ is supported in $J_\pm(x)$. For $\Re(\alpha)\geq d$, $R_\pm(\alpha,x)$ is given by \[R_\pm(\alpha,x)(y)=\casedist{c_\alpha\Gamma_x(y)^{\frac{\alpha-d}{2}}}{y\in J_\pm(x)}{0}{y\notin J_\pm(x)}.\]
Moreover, $R_\pm(0,x)=\delta_x$ is the Dirac distribution at $x$ (see \cite[Proposition 1.4.2]{BGP}).

We now turn to the second class of objects we want to define in this section: the Hadamard coefficients.
The motivation for defining the Hadamard coefficients is that we want to have the formal equality
\[P\insum{k}V^{k}_x R_\pm(2k+2,x)=\delta_x\]
for some smooth sections $V^{k}(x)$. Formally, this would mean that the infinite sum should be a fundamental solution for $P$ (i.e. the Schwartz kernel of $G^\pm$). This requires us to investigate the application of $P$ to a product involving Riesz distributions.
\begin{defprop}
	\label{dfrho}
	\label{Prho}
	Define for $V\in \Gamma(E\otimes E_x)$:
	\[\rho_x V:= \nabla_{\grad\Gamma_x}V-\left(\frac{1}{2}\square\Gamma_x-d\right)V.\]
	Then we have for any $\alpha\in \C\O$, $x\in U$ and $V\in \Gamma(E|_U\otimes E^*_x)$
	\[P(VR_\pm(\alpha+2,x))=(PV)R_\pm(\alpha+2,x)-\tfrac{1}{\alpha}((\rho_x-\alpha)V)R_\pm(\alpha,x)\]
\end{defprop}
The reader may immediately forget the definition of $\rho_x$ and only remember the second formula.

We are now ready to define the Hadamard coefficients (see \cite[Definition 2.2.1 and Proposition 2.3.1]{BGP}):
	\begin{defprop}
		\label{dfV}
		For every $x\in U$ there is a unique family of sections $V^{k}_x\in \Gamma(E\otimes E^*_x|_U)$ (indexed by $k\in \N$) that satisfies the transport equations 
		\[(\rho_x-2k) V^{k}_x=2kP V^{k-1}_x\]
		(for $k=0$, the right hand side is set to $0$) subject to the initial condition
		\[V^{0}_x(x)=1.\]
		These sections are called the Hadamard coefficients (for $P$). The map $(y,x)\mapsto V^{k}_x(y)$ is a smooth section in $E\boxtimes E^*|_{U\times U}$.
	\end{defprop}
This definition is designed to make the following formal calculation work:
\begin{align*}
	&P\insum{k}V^k_xR(2k+2,x)\\
	&=\lim\limits_{\alpha\rightarrow 0}P(V^0_xR(2+\alpha,x)+\sum\limits_{k=1}^\infty P(V_kR(2k+2,x))\\
	&=V^0_xR_0 +\lim\limits_{\alpha\rightarrow 0}\tfrac{1}{\alpha}(\rho_xV^0_x)R(\alpha,x) +(PV_0)R_\pm(2,x)+\\&\ \sum\limits_{k=1}^\infty (PV_k)R(2k+2,x)-\tfrac{1}{2k}((\rho_x-2k)V)R_\pm(2k,x)\\
	&=V^0_x(x)\delta_x+0+\sum\limits_{k=1}^\infty \tfrac{1}{2k}(2kPV^{k-1}_x-(\rho_x-2k) V^{k})R(2k,x)\\
	&=\delta_x+\sum\limits_{k=1}^\infty0\\
	&=\delta_x
\end{align*}
However, in general the remainder $(PV^{N}_x)R_\pm(2N+2,x)$, which we get when replacing the summation limit $\infty$ with some finite number $N$ may not vanish for $N\rightarrow \infty$. All we get is that the remainder becomes arbitrarily differentiable for $N$ large enough, i.e. we get an asymptotic expansion in differentiability order.
This is the motivation for the definition of the Hadamard coefficients and one of the key results that this paper will build on:
We have the asymptotic expansion
\[G^\pm_x|_U\sim\insum{k}V^{k}_xR_\pm(2k+2,x).\]
This expansion is shown in \cite [Proposition 2.5.1]{BGP} (a pedantic reader may note that the expansion in \cite{BGP} is shown under different assumptions on the neighborhood $U$. We refer this reader to \cite[Theorem 3.4]{Ron} for a version with this choice of $U$).

\begin{rem}
	Note that both $G^\pm_x$ and $R_\pm(2k+2,x)$ are supported inside $J_\pm(x)$ hence the same is true for the remainder term in the expansion. Thus all derivatives that the remainder term has must be vanishing at the boundary of $J_\pm(x)$, in particular at $x$. This means that differentiability of the remainder leads to decay near $x$ of the corresponding order (see \cite[Theorem 2.5.2]{BGP} for a detailed statement).
	
	There is another propagator associated to $P$, known as the Feynman propagator, that also has a similar asymptotic expansion involving the Hadamard coeficients. However, in the Feynman case nothing needs to vanish anywhere, so differentiability will not transform into decay estimates. That is the reason why we need the causal propagator in the approach presented here.
\end{rem}

Our construction depends on the choice of neighborhood $U$, and so do the values of the Hadamard coefficients. However, when replacing $U$ by a subset, the values stay the same. In particular, the values of the Hadamard coefficients on the diagonal (i.e. $V^k_x(x)$) do not depend on the choice of $U$.

For technical reasons, we will be looking not at the two Green's operators individually, but at their difference, sometimes referred to as the causal propagator.
\begin{df}
	\label{dfcp}
	Set
	\[G:=G^+-G^-,\]
	\[R(\alpha):=R_+(\alpha)-R_-(\alpha)\]
	and
	\[R(\alpha,x):=R_+(\alpha,x)-R_-(\alpha,x).\]
\end{df}
Taking the difference of the asymptotic expansions for $G^\pm_x$, we obtain
\[G_{x}|_U\sim\insum{k}V^{k,U}_xR(2k+2,x).\]
The advantage of taking this difference is that some singularities on the diagonal cancel out, which will allow us to restrict to a timelike curve.
\subsection{Wavefronts}
Wavefront calculus is a tool for studying singularity structures of distributions. It allows one to perform operations that do not work for arbitrary distributions on those distributions that do not have the wrong singularities. In our context, it will be used to restrict Riesz distributions and Green's kernels to a timelike curve.

The wavefront set is defined via its complement: For a distribution $u$ on $\R^n$, a pair $(x,\xi)\in\R^n\times \R^n\O$ is not in the wavefront set of $u$ if and only if for every cut-off function $\phi$ with $\phi(x)=1$, the Fourier transform $\F(\phi u)$ is rapidly decaying (i.e. the product with any polynomial is bounded) in a conical neighborhood of $\xi$. This can be roughly interpreted as "$u$ is smooth in the direction $\xi$ near the point $x$". Thus the wavefront set captures in which direction singularities occur at each point. As the wavefront set is locally defined and transforms well under diffeomorphisms, it can also be defined on arbitrary manifolds as a subset of the cotangent bundle.

The reason we need wavefronts is the following fact: If the wavefront set of a distribution doesn't intersect the conormal of a submanifold, then the distribution can be restricted to that submanifold. This is a special case of a more general theorem about pullbacks, as a restriction is just the pullback by the inclusion map. What does it mean that the pullback "can be defined"? It means that the space of distributions with wavefronts contained in a specified closed conic set that does not intersect the normal bundle can be equipped with reasonable seminorms, such that the pullback continuously extends to this space (see \cite[Proposition 5.1]{BDH}). This continuity ensures that the extension of the pullback has the properties that one would expect from a pullback (e.g. $w^*(f\cdot g)=(f\circ w)\cdot w^*(g)$ for a smooth function $w$).

Our goal is to restrict the Riesz distributions and Green's kernels to a timelike curve $w$, i.e. we want to consider the pullbacks $w^*(R(\alpha,x))$ and $w^*(G_x)$. As the conormal bundle of a timelike curve contains only spacelike vectors, we need to show that these are not contained in the corresponding wavefront sets. Moreover, in order to work with the pulled-back Riesz distributions, we need to know that they are still holomorphic in $\alpha$. We sketch the proof in the following (see \cite[Section 4.1]{Ron} for full details). Note that the individual Riesz distributions $R^\pm(\alpha,x)$ do not satisfy this. They have problematic singularities at the origin that only cancel when taking the difference. This is the reason why we always work with differences of Riesz distributions and Green's operators.

\begin{thm}
	\label{WFRiesz}
	The family of Riesz distributions on Minkowski space is contained in and holomorphic in the space of Distributions whose wavefront does not contain spacelike Vectors and away from zero only contains multiples of $d\gamma$ at points where $\gamma$ vanishes.
\end{thm}
\begin{rem}
	A more precise characterization of the wavefront at zero can be obtained but will not be needed here.
\end{rem}
\begin{proof}[proof sketch]\hfill
	\begin{itemize}
	\item On Minkowski space, define a signed cut-off $\sigma$ to be $1$ on future causal vectors, $-1$ on past causal vectors and $0$ on spacelike vectors. For a function $f$ on $\R$, we can consider
	\[\sigma\gamma^*(f)=\sigma\cdot (f\circ \gamma).\]
	For $f(x)=c_\alpha x^\alpha$ we obtain the riesz distributions in this way.
	\item The wavefront set away from zero can be obtained directly via wavefront calculus, as $\Gamma$ is a submersion there. It thus remains to control the wavefront at $0$.
	\item Let $\xi$ be a spacelike vector (we want to show that this is not in the wavefront of $\sigma\gamma^*f$). Define
	\[O_\xi(x):=2\<x,\xi\>y_\xi-x\]
	(for this proof $\<,\>$ denotes the Riemannian scalar product on $\R^d$ rather than the Minkowski product, as the former is used to define the fourier transform).
	One can check the following identities:
	\begin{align*}
		\<O_\xi(x),\xi\>&=\<x,\xi\>\\
		\gamma(O_\xi(x))&=\gamma(x)\\
		|\det(O_\xi)|&=1\\
		\sigma(O_\xi(x))&=-\sigma(x)\\
	\end{align*}
	\item Choose a cut-off $\psi_\xi$ that is also preserved by $O_\xi$. First doing a substitution and then using that $O_\xi$ preserves everything except timeorientation, we obtain for any $\lambda\in \R$:
	\begin{align*}
		\F(\psi_\xi\sigma\gamma^*(f))(\lambda\xi)&=\int\limits_{\R^d}e^{-i\lambda\<x,\xi\>}\psi_\xi(x)\sigma(x)f(\gamma(x))dx\\
		&=\int\limits_{\R^d}|det(O_\xi)|e^{-i\lambda\<O_\xi(x),\xi\>}\psi_\xi(O_\xi(x))\sigma(O_\xi(x))f(\gamma(O_\xi(x)))dx\\
		&=-\int\limits_{\R^d}e^{-i\lambda\<x,\xi\>}\psi_\xi(x)\sigma(x)f(\gamma(x))dx\\
		&=-\F(\psi_\xi\sigma\gamma^*(f))(\lambda\xi)\\
		&=0
	\end{align*}
	In particular, the fourier transform of $\psi_\xi \sigma\gamma^*f$ is rapidly decaying on multiples of $\xi$ (being constantly zero).
	\item One can show that the rapid decay is preserved under a change in the cut-off function. Thus spacelike vectors are not contained in the wavefront set of $\sigma\gamma^*f$.
	\item Keeping track of wavefront norms along the way, one can show that in fact the map $\sigma\gamma^*$ maps distributions on $\R$ with wavefront at zero continuously to distributions on $\R^d$ with non-spacelike wavefront.
	\item The Riesz distributions are the image under this map of the family of distributions that is defined by extending $c_\alpha \mathbbm1_{x>0}x^\alpha$ holomorphically in $\alpha$. As continuous linear maps preserve holomorphicity, the Riesz distributions are holomorphic in the target space.
	\end{itemize}
\end{proof}
Combining this result with the pullback theorem from wavefront calculus, we obtain:
\begin{cor}
	For any timelike curve $w$ in $\R^d$, $w^*(R(\alpha))$ is well defined and holomorphic in $\alpha$ (as distributions).
\end{cor}
By diffeomorphism invariance, we get the same thing for $R(\alpha,x)$, provided that the image of the path $w$ under normal coordinates around $x$ is timelike. If $w$ is a timelike curve through $x$, this will at least be true locally around $x$. Away from $x$, the path will be in the timelike future or past of $x$, where the Riesz distributions are smooth anyways. Thus we obtain
\begin{cor}
	\label{Rieszpullback}
	For any timelike curve $w$ in $U$, $w^*(R(\alpha,w(0)))$ is well defined and holomorphic in $\alpha$ (as distributions).
\end{cor}
Using the Hadamard expansion, this carries over to the Green's kernels (as we can always pull back a continuous remainder term).

\section{Constructing an action-like function}
\label{caf}
We now seek to integrate Riesz distributions and Green's kernels along a timelike curve $w$, paired with a test function $f$, which we rescale in order to study the behavior around $0$.
We start by fixing some notation:
\begin{gn}
	\label{dfw}
	For the remainder of this section, let $I\subseteq\R$ be an open interval containing $0$ and $x\in U$. Furthermore, let $w\colon I\rightarrow U$ be a future oriented timelike curve with $w(0)=x$. In section \ref{Hadamextract}, $w$ will additionally be assumed to be a timelike unit speed geodesic, but for now we keep it more general. Define 
	\[\nu_w(t):=\frac{\Gamma_x(w(t))}{t^2}.\]
\end{gn}
Note that $\nu_w(t)$ extends smoothly to $0$, as $\Gamma_x$ vanishes of order $2$ at $x$. We have
\[\nu_w(0)=\gamma(w'(0)),\]
as can be checked in normal coordinates around $x$.
Before we proceed in a more rigorous manner, let us do an informal calculation to motivate what we do. Assume that $k$ is large enough such that $R(2k+2,x)$ is an actual function. Assume furthermore that $w$ is a unit speed geodesic, so that $\nu_w=1$. We taylor expand 
\[V^k_x(w(t))\sim\insum{l}v_l t^l.\]
We pair the pullback along $w$ with a test function $f$ supported in $I$ that we scale down in order to analyze the behavior near the diagonal. This gives us (at least formally)
\begin{align*}
	w^*(V^k_xR(2k+2,x))[f(\tfrac{\cdot}{s})]&=\int\limits_I V^k_x(w(t))R(2k+2,x)(w(t))f(\tfrac{t}{s})dt\\
	&\sim \insum{l}\int\limits_I v_lt^l c_{2k+2} \sign(t)|t|^{2k+2-d}f(\tfrac{t}{s})dt\\
	&=\insum{l}v_l c_{2k+2} \int\limits_\R t^l  \sign(t)|t|^{2k+2-d}f(\tfrac{t}{s})dt\\
	&=\insum{l}v_l c_{2k+2}\int\limits_\R s(st)^l  \sign(st)|st|^{2k+2-d}f(t)dt\\
	&=\insum{l}v_l c_{2k+2}s^{l+2k+3-d} \int\limits_\R t^l  \sign(t)|t|^{2k+2-d}f(t)dt\\
\end{align*}
Summing this over $k$, we obtain a small $s$ asymptotic expansion for $w^*(G_x)[f(\tfrac{\cdot}{s})]$ in terms of the Taylor coefficients (i.e. the derivatives in time direction) of the Hadamard coefficients. Note that the remaining integral depends neither on $s$ nor on the geometry of the manifold.
If we assume that $f$ is odd, the integral will vanish for all odd $l$ and for even $l$ we obtain
\[\int\limits_\R t^l  \sign(t)|t|^{2k+2-d}f(t)dt=2\intP f(t)t^{l+2k+2-d}dt.\]
This is (one value of) the mellin transform of $f$. To handle smaller values of $k$, we need a variant of the Mellin transform where some poles are removed:
\begin{df}
	\label{dfMprime}
	In the following we abbreviate:
	\[\M'(g):=\frac{\M(g)}{\Gamma(\frac{\cdot+1}{2})},\]
	where $\M(g)$ denotes the Mellin transform of $g$, i.e. the meromorphic continuation of 
	$\alpha \mapsto \intP g(x)x^{\alpha-1}dx$ and $\Gamma$ denotes the gamma function.
\end{df}
For odd smooth $g$, the Gamma function in the denominator will cancel with poles of the Mellin transform, allowing us to evaluate the modified Mellin transform everywhere. 
As we have seen above, the distinction between even and odd functions will also play a role.
\begin{df}
	\label{evenodd}
	For a function $f$ on $\R$, let 
	\[(f)_{odd}(t):=\frac{1}{2}(f(t)-f(-t))\]
	and
	\[(f)_{even}(t):=\frac{1}{2}(f(t)+f(-t))\]
	denote the odd and even part of $f$. If $f$ is defined only on a subset of $\R$, extend it by $0$.
\end{df}
We now want to make the above considerations more rigorous.
We start by evaluating the pull-back of a single summand $V_x^kR(2k+2,x)$. 
\begin{prop}
	\label{WRint}
	For any $g\in C_c^\infty(I)$ and $\alpha\in \C$, the pullback $w^*({V^k_x}R(\alpha,x))$ is well-defined and we have:
	\[w^*({V^k_x}R(\alpha,x))[g]=\frac{2^{2-\alpha}\pi^\frac{2-d}{2}}{\Gamma(\frac{\alpha}{2})}\M'\left(\left(\nu_w^\frac{\alpha-d}{2}({V^k_x}\circ w)g\right)_{odd}\right)({\alpha-d+1}).\]
\end{prop}
\begin{rem}
	Read this as
	\[\intR R(\alpha,x)(w(t)){V^k_x}(w(t))g(t)dt=\intP c_\alpha \nu_w^\frac{\alpha-d}{2} t^{\alpha-d} ({V^k_x}(w(t))g(t))_{odd} dt.\]
\end{rem}
\begin{proof}
	Well-definedness of the pull-back follows from Corollary \ref{Rieszpullback}.
	For $\Re(\alpha)>d$ (where Riesz distributions are functions), we have
	\[R(\alpha,x)(w(t))=c_\alpha \Gamma_x(w(t))^\frac{\alpha-d}{2}(\mathbbm{1}_{w(t)\in J_+(x)}-\mathbbm{1}_{w(t)\in J_-(x)})=c_\alpha \nu^\frac{\alpha-d}{2} |t|^{\alpha-d}\sign(t).\]
	Thus we get
	\begin{align*}
		&w^*({V^k_x}R(\alpha,x))[g]\\
		&=\int\limits_I {V^k_x}(w(t))R(\alpha,x)(w(t))g(t)dt\\
		&=c_{\alpha}\int\limits_I \sign(t){V^k_x}(w(t))\nu_w(t)^\frac{\alpha-d}{2}|t|^{\alpha-d}g(t)dt\\
		&=2c_\alpha\intP\left(\nu_w^\frac{\alpha-d}{2}({V^k_x}\circ w)g\right)_{odd}(t) t^{\alpha-d}dt\\
		&=2c_\alpha \M\left(\left(\nu_w^\frac{\alpha-d}{2}({V^k_x}\circ w)g\right)_{odd}\right)({\alpha-d+1})\\
		&=\frac{2^{2-\alpha}\pi^\frac{2-d}{2}}{\Gamma(\frac{\alpha}{2})} \M'\left(\left(\nu_w^\frac{\alpha-d}{2}({V^k_x}\circ w)g\right)_{odd}\right)({\alpha-d+1}).
	\end{align*}
	As both sides are meromorphic in $\alpha$, they coincide for all $\alpha$.
\end{proof}
In order to retain more information after distributional evaluation, we equip the function we evaluate at with an additional parameter $s$.
\begin{gn}
	\label{dff}
	Let $f\in C_c^\infty(\R)$ be an odd function. For $s>0$ and $t\in \R$, define \[f_s(t):=f\left(\frac{t}{s}\right).\] 
\end{gn}
This $f$ corresponds roughly to (the derivative of the Fourier transform of) the cutoff function in the spectral action, $s$ corresponding to the cutoff parameter there. The $f$ here being odd corresponds to that in the spectral action being even.
The assumption that $f$ is odd is not crucial for the following considerations, but it makes things somewhat simpler and adding an even part would not yield anything helpful for our purposes.

Testing a distribution against $f_s$ for small $s$ corresponds to investigating its behaviour around $0$. Multiplying with $f_s$ and taking the Mellin transform does something similar. This is investigated in the following two lemmas.
\begin{lemma}
	\label{IntRem}
	If $h$ is a $C^{k+1}$ function whose first $k$ derivatives at $0$ vanish and $\alpha\in \C$ with $\Re(\alpha)\geq -k$, we have
	\[|\M(hf_s)(\alpha)|=O(s^{k+\Re(\alpha)})\]
	and
	\[\intR h(t)f_s(t)dt=O(s^{k+1}).\]
\end{lemma}
\begin{rem}
	This translates error terms in taylor series or asymptotic expansions in differentiability orders (if all continuous derivatives vanish at 0) to error terms for asymptotic expansions in $s$. 
\end{rem}
\begin{proof}
	By Taylor's theorem (with vanishing Taylor series),
	\[h(t)= E(t)t^k\]
	for some function $E$ that is bounded by a constant $C$ on a compact set containing $\supp(f_s)$ for any $|s|<1$. Thus we have for $|s|<1$:
	\begin{align*}
		| \M (hf_s)(\alpha)|&=|\M (Ef_s)(\alpha+k)|\\
		&\leq\intP C|t^{k+\alpha-1}f\left(\tfrac{t}{s}\right)|dt\\
		&=|s^{k+\alpha}|\intP\frac{1}{s} C|\left(\tfrac{t}{s}\right)^{k+\alpha-1}f\left(\tfrac{t}{s}\right)|dt\\
		&=C|s^{k+\alpha}|\intP |t^{k+\alpha-1}f(t)|dt\\
		&=O(s^{k+\Re(\alpha)}).
	\end{align*}
	
	The second part of the claim follows either by doing an analogous calculation or by observing that
	\[\intR h(t)f_s(t)dt=\intP(h(t)-h(-t))f_s(t)dt= \M ((h-h(-\cdot))f_s)(1).\qedhere\]
\end{proof}
\begin{lemma}
	\label{MsExp}
	Let $h$ be a smooth function defined on $I$ and $\alpha\in \C$. Then
	\[\M'( (hf_s)_{odd})(\alpha)\stackrel{s\rightarrow 0}{\sim}\insum{k}\frac{k!}{(2k)!}\binom{\frac{\alpha+2k-1}{2}}{k}h^{(2k)}(0) \M'(f)(\alpha+2k)s^{\alpha+2k}.\]
\end{lemma}
\begin{proof}	
	Write $t^k$ for the function $t\mapsto t^k$. For $k\in \N$ and $\Re(\alpha)\geq 0$, we have 
	\begin{align*}
		\M' (t^k f_s)(\alpha)&=\frac{1}{\Gamma(\frac{\alpha+1}{2})}\intP t^{\alpha+k-1}f\left(\frac{t}{s}\right)dt\\
		&=\frac{1}{\Gamma(\frac{\alpha+1}{2})}s^{\alpha+k}\intP\frac{1}{s}\left( \frac{t}{s}\right)^{\alpha+k-1}f\left(\frac{t}{s}\right)dt\\
		&=\frac{1}{\Gamma(\frac{\alpha+1}{2})}s^{\alpha+k}\intP t^{\alpha+k-1}f(t)dt\\
		&= \frac{1}{\Gamma(\frac{\alpha+1}{2})}\M(f)(\alpha+k)s^{\alpha+k}.\\
	\end{align*}
	By analytic continuation, this holds for all $\alpha\in \C$. Now fix $\alpha\in \C$. As $f$ is odd, we have $(hf_s)_{odd}=(h)
	_{even}f_s$.
	We write $(h)_{even}$ in a Taylor expansion
	\[(h)_{even}(t)=\sum\limits_{k=0}^N\frac{1}{(2k)!}h^{(2k)}(0)t^{2k}+ E(t),\]
	where the first $2N$ derivatives of $E$ at $0$ vanish.
	We obtain for $N\in \N$ such that $2N\geq-\Re(\alpha)$ (using Lemma \ref{IntRem} for the remainder term) :
	\begin{align*}
		&\M'((hf_s))(\alpha)\\
		&=\sum\limits_{k=0}^N\frac{1}{(2k)!}h^{(2k)}(0)  \left(\frac{\Gamma(\frac{\cdot+2k+1}{2})\M (t^{2k}f_s)}{\Gamma(\frac{\cdot+2k+1}{2})\Gamma(\frac{\cdot+1}{2})}\right)(\alpha)
		+ \M' ( Ef_s)(\alpha)\\
		&=\sum\limits_{k=0}^N\frac{k!}{(2k)!}h^{(2k)}(0)  \left(\frac{\Gamma(\frac{\cdot+1}{2}+k)}{k!\Gamma(\frac{\cdot+1}{2})}\frac{\M (f_s)(\cdot+2k)}{\Gamma(\frac{\cdot+2k+1}{2})}\right)(\alpha)
		+ O(s^{2N+\Re(\alpha)})\\
		&=\sum\limits_{k=0}^N\frac{k!}{(2k)!}\binom{\frac{\alpha+2k-1}{2}}{k}h^{(2k)}(0) \M'(f)(\alpha+2k)s^{\alpha+2k}+ O(s^{2N+\Re(\alpha)}).\\
	\end{align*}
	As $N$ was arbitrary, this gives the desired result.
\end{proof}

Putting together what we have done, we obtain this sections main result:
\begin{thm}
	\label{intexp}
	\label{dfan}
	$w^*(G_x)$ is well-defined and we have
	\begin{align*}
		w^*(G_x)[f_s]\stackrel{s\rightarrow 0}{\sim}&\insum{k}\insum{n}a(k,n)\left(\nu_w^{k-\tfrac{d}{2}+1}V^{k}_x\circ w\right)^{(2n)}(0)s^{2k+2n+3-d}
	\end{align*}
	with 
	\[a(k,n):=a(k,n,f,d):=\frac{\pi^\frac{2-d}{2}n!}{4^kk!(2n)!} \binom{k+n+1-\frac{d}{2}}{n} \M'(f)({2k+2n+3-d})\]
	being coefficients independent of the geometry (except through the dimension).
\end{thm}

\begin{proof}
	 Let $s\in (0,1)$ be small enough that $f_s$ is supported in $I$. Let $m\in \N$ be arbitrary and let $N\in\N$ be large enough such that
	\[F:=G_x-\sum\limits_{k=0}^N V^k_xR(2k+2,x)\in C^m(U).\]
	As both the sum and the remainder have well-defined pull-backs under $w$, the same is true for $G_x$. As both $G_x$ and all $R(2k+2,x)$ vanish outside of $J(x)$, the same is true for $F$ and thus also for its derivatives. As $x$ is in the boundary of $J(x)$, this means that all derivatives of order up to $m$ of $F$ vanish at $x$. Therefore the first $m$ derivatives of $w^*F=F\circ w$ at 0 vanish.
	By Lemma \ref{IntRem}, we can conclude that
	\[|w^*(F)[f_s]|=\Big|\intR F\circ w(t)f_s(t)dt\Big|=O(s^m).\]
	For each of the summands, we can apply Proposition \ref{WRint} and Lemma \ref{MsExp} to obtain for $k<N$ and $N'$ large enough:
	\begin{align*}
		&w^*(V^k_xR(2k+2,x))[f_s]\\
		&=\frac{\pi^\frac{2-d}{2}}{4^kk!} \M'\left(\left((\nu_w^{k-\tfrac{d}{2}+1}V^k_x\circ w)f_s\right)_{odd}\right)({2k+3-d})\\
		&= \sum\limits_{n=0}^{N'}\frac{\pi^\frac{2-d}{2}}{4^kk!}\frac{n!}{(2n)!}\binom{k+n+1-\frac{d}{2}}{n}\left(\nu_w^{k-\tfrac{d}{2}+1}V^k_x\circ w\right)^{(2n)}(0)\\
		&\ \ \cdot\M'(f)({2k+2n+3-d})s^{2k+2n+3-d}+O(s^m).\\
	\end{align*}
	We then have
	\begin{align*}
		w^*(G_x)[f_s]&=\sum\limits_{k=0}^Nw^*(V^k_x R(2k+2,x))[f_s]+w^*(F)[f_s]\\
		&=\sum\limits_{k=0}^N\sum\limits_{n=0}^{N'}\frac{\pi^\frac{2-d}{2}n!}{4^kk!(2n)!} \binom{k+n+1-\frac{d}{2}}{n}\left(\nu_w^{k-\tfrac{d}{2}+1}V^k_x \circ w\right)^{(2n)}(0)\\&\cdot \M'(f)({2k+2n+3-d})s^{2k+2n+3-d}+O(s^m).
	\end{align*}
	As $m$ was arbitrary, this concludes the proof.
\end{proof}

\begin{rem}
	Keeping track of the remainder terms in this section, one can show that this asymptotic expansion is sufficiently uniform that it still holds after integrating over compact sets (see \cite[Theorem 4.2.15]{Ron} for technical details).
\end{rem}

\section{Extracting Hadamard coefficients}
\label{Hadamextract}
We now seek to use the result of the previous chapter to obtain the Hadamard coefficients on the diagonal, i.e. $V^{k}_x(x)$. The problem with the asymptotic expansion we derived is that we cannot distinguish between the contribution coming from the $k$-th Hadamard coefficient and that from the $2j$-th derivative of the $(k-j)$-th coefficient. This problem does not arise from the choice of action-like function we construct, but is inherent in the asymptotic expansion we are using.

The problem is similar to that of determining the coefficients of a power series 
\[\insum{k}a_k(x)x^k\]
where the coefficients are themselves depending on $x$. Not even the values of the coefficients at $x=0$ (except for the first) are determined by this expansion, as we may always add some function $g$ to one of the coefficients and subtract $xg$ from the previous one without changing the sum.

This problem also occurs with the Hadamard expansion: As
\[\Gamma R(2k,x)=2k(2k-d+2)R(2k+2,x),\]
we may (for $k\notin \{0, \frac{d}{2}-1\}$) add some $g$ to $V^{k}_x$ and subtract $\frac{\Gamma_x g}{2k(2k-d+2)}$ from $V^{k+1}_x$ without changing the sum in the asymptotic expansion of $G_x$. Thus $V^{0}$ and $V^{\frac{d}{2}-1}$ (in case $d$ is even) are the only Hadamard coefficients whose value at the diagonal we may hope to obtain by using only this expansion (and those we actually do obtain from our procedure).

We thus need to include some extra information about the Hadamard coefficients, if we want to succeed. We do this  by looking at the Hadamard expansion of $P-z$ for variable $z\in \C$ and expressing the Hadamard coefficients of $P-z$ in terms of those of $P$. An alternative approach would be to use the asymptotic expansion for powers of $G^\pm$ instead. 
\begin{rem}
	For even dimension, the $k=\tfrac{d}{2}-1$-coefficient is exempt from this problem, as the factor we would divide by is $0$ in that case. In fact that coefficient can then be obtained directly either from the $s^1$ coefficient in the expansion of Theorem \ref{intexp} or more directly by evaluating the Green's kernel at $y\in J_+(x)$ and taking the limit $y\rightarrow x$. This can then be used to obtain a formula for the first $\tfrac{d}{2}-1$ coefficients by "shifting up" those coefficients to the $\tfrac{d}{2}-1$-slot, where we can read them off (see \cite[Section 5.1]{Ron}). Unfortunately, there is no way to "shift down" the coefficients, so for the higher order coefficients, the more general approach described in the following is required.
\end{rem}
The key insight we will use is the following:

\begin{prop}
	\label{Vz}
	We have
	\[V^{k}_{x}(z)=\sum\limits_{m=0}^k\binom{k}{m}z^mV^{k-m}_x.\]
\end{prop}
\begin{rem}
	The corresponding formula for the heat coefficients
	\[a_k(\Delta-z)=\sum\limits_{m=0}^k\frac{1}{m}z^ma_{k-m}(\Delta)\]
	can be shown analogously. Alternatively, that formula can also be deduced by Taylor expanding $e^{tz}$ and multiplying out the expansions in 
	\[e^{-t(\Delta-z)}=e^{tz}e^{-t\Delta}.\]
\end{rem}
\begin{proof}
	Let
	\[V_k(z):=\sum\limits_{m=0}^k\binom{k}{m}z^mV^{k-m}_x\]
	We need to show that $V_0(z)(x,x)=1$ and the $V_k(z)$ satisfy the transport equations
	\[(\rho_x-2k)V_k(z)=2k(P-z)V_{k-1}(z)\]
	with $\rho$ as in \ref{dfrho}.
	The former holds, as $V_0(z)=V^0_x$. For the latter, we calculate using the transport equation for $P$:
	\begin{align*}
		&(\rho_x-2k)V_k(z)\\
		&=\sum\limits_{m=0}^k\binom{k}{m}z^m(\rho_x-2k)V^{k-m}_x\\
		&=\sum\limits_{m=0}^k\binom{k}{m}z^m((\rho_x-2(k-m))V^{k-m}_x-2mV^{k-m}_x)\\
		&=\sum\limits_{m=0}^k\binom{k}{m}z^m2(k-m)PV^{k-m-1}_x-\sum\limits_{m=0}^k\binom{k}{m}z^m2mV^{k-m}_x\\
		&=\sum\limits_{m=0}^k\binom{k-1}{m}z^m2kPV^{k-m-1}_x-\sum\limits_{m=1}^k\binom{k-1}{m-1}z^m2kV^{k-m}_x\\
		&=\sum\limits_{m=0}^{k-1}\binom{k-1}{m}z^m2kPV^{k-m-1}_x-\sum\limits_{m=0}^{k-1}\binom{k-1}{m}z^{m+1}2kV^{k-m-1}_x\\
		&=2k(P-z)\sum\limits_{m=0}^{k-1}\binom{k-1}{m}z^mV^{k-m-1}_x\\
		&=2k(P-z)V_{k-1}(z).
	\end{align*}
	The $V_k(z)$ thus satisfy the transport equations for $P-z$, so they are the desired Hadamard coefficients.
\end{proof}
This proposition allows us to perform the previous analysis for $P-z$ and still get a result that depends on the Hadamard coefficients of $P$. The additional dependence on the parameter $z$ then allows us to "invert" the formula for the expansion coefficients and obtain a formula for the Hadamard coefficients. To facilitate the inversion process, we will assume that $w$ is a unit speed geodesic ($\nu_w=1$).
\begin{defprop}
	\label{Expwithz}
	Assume for this section that $\nu_w=1$.

	For a function $F$ with an asymptotic expansion \[F\sim \sum\limits_{p\in P}a_p p\] with $P$ a set of monomials in one or more variables, we denote the coefficients of each $p$ by
	\[F[[p]]:=a_p.\]
	Define
	\[L(s,z):=w^*(G_{P-z,x})[f_s]\]
	to be the action-like function constructed in the previous section, with the Green's operator for $P-z$ instead of that for $P$.
	Under the assumptions of Theorem \ref{intexp}, we have for any $z\in \C$ the following asymptotic expansion in $s$:
	\begin{align*}
		L(s,z)\stackrel{s\rightarrow 0}{\sim}\insum{k}\insum{m} L_{k,m} z^m s^{2k+2m+3-d}
	\end{align*}
	with 
	\[L_{k,m}:=L[[ s^{2k+2m+3-d}]][[z^m]]=\sum\limits_{l=0}^k\binom{l+m}{m} a(l+m,k-l)\left(V^{l}_x\circ w\right)^{(2k-2l)}(0)\]
	for the coefficients $a$ from Theorem \ref{intexp}.
\end{defprop}
\begin{rem}
	The $L_{k,m}$ can be recovered from $L$, e.g. via
	\[L[[ s^{2k+2m+3-d}]][[z^m]]=\frac{1}{m!(2k+2m)!}\tfrac{\partial}{\partial z}^m\tfrac{\partial}{\partial s}^{2k+2m}s^{d-3} L(s,z).\] 
	(Differentiating the asymptotic expansion actually requires further properties of the remainder terms, but in our case all remainder terms are of the form $\intR E(x)f_s(x)dx$ for functions $E$ that are sufficiently often differentiable with vanishing derivatives at $0$. These remainder terms do not lose too many $s$-powers when differentiated.)
	Our goal will thus be to express the Hadamard coefficients in terms of the $L_{k,m}$.
\end{rem}
\begin{proof}
	Using Theorem \ref{intexp} and Proposition \ref{Vz} before doing substitutions ($l=k-m$ and $K=l+n$) to rearrange the summands, we obtain
	\begin{align*}
		L(s,z)=&w^*(G_{P-z,x})[f_s]\\
		&\stackrel{s\rightarrow 0}{\sim}\insum{k}\insum{n}a(k,n)\left(V^{k}_x(z)\circ w\right)^{(2n)}(0)s^{2k+2n+3-d}\\
		&=\insum{k}\insum{n}\sum\limits_{m=0}^k\binom{k}{m}z^m \left(V^{k-m}_x\circ w\right)^{(2n)}(0)a(k,n)s^{2k+2n+3-d}\\
		&=\insum{l,m,n}\binom{l+m}{m}z^m \left(V^{l}_x\circ w\right)^{(2n)}(0)a(l+m,n)s^{2(l+m)+2n+3-d}\\
		&=\insum{K,m}\sum\limits_{l=0}^K\binom{l+m}{m} a(l+m,K-l)\left(V^{l}_x\circ w\right)^{(2K-2l)}(0)z^m s^{2K+2m+3-d}\\		
	\end{align*}
\end{proof}

For fixed $k$, this gives us a system of linear equations (one for each $m$) relating the $L_{k,m}$ to the derivatives of the Hadamard coefficients. Solving these equations amounts to finding a left inverse for the matrix with entries
\[\alpha(k)_{ml}=\binom{l+m}{m} a(l+m,k-l).\]
In fact, since we don't care about the derivatives, it suffices to find the last row of a left inverse:
\begin{cor}
	Let $k\in \N$ Assume $(q_{km})_{m\leq N}$ for $N\in \N$ are coefficients such that
	\[\sum\limits_{m=0}^Nq_{km}\alpha(k)_{ml}=\delta_{kl}\]
	for $(\alpha(k)_{ml})$ as defined above and all $l\in \N$. Then we have
	\[V^{k}_x(x)=\sum\limits_{m=0}^N q_{km}L_{k,m}.\]
\end{cor}
\begin{proof}
	We have
	\begin{align*}
		\sum\limits_{m=0}^N q_{km}L_{k,m}
		&=\sum\limits_{l=0}^k \sum\limits_{m=0}^N q_{km}\alpha(k)_{ml}\left(V^{l}_x\circ w\right)^{(2k-2l)}(0)\\
		&=\sum\limits_{l=0}^k \delta_{kl}\left(V^{l}_x\circ w\right)^{(2k-2l)}(0)\\
		&=V^{k}_x(x).
	\end{align*}
\end{proof}

This left inverse is far from unique. A possible choice of left inverse can be found by inverting any quadratic block in this matrix. For a more versatile result, we do not commit to a fixed block but compute the inverse for the block where $m$ starts at some offset $o$. As the computation of these inverses is purely combinatoric and of little significance to understanding the overall argument, we omit it here, refer the interested reader to \cite[Section 5.2]{Ron} and only give the final result (\cite[Theorem 5.2.12]{Ron}):
\begin{thm}
	\label{main}
	Assume that $\M'(f)$ is non-zero at integers and let $o\in \N_0$.
	We have
	\[V^k_x(x)=\sum\limits_{m=0}^k q(k,m,o)L_{k,m+o}\]
	with 
	\[q(k,m,o):=\frac{\pi^{\frac{d}{2}-1}4^{k+m+o}(m+o)!k!}{\M'(f)({2k+2m+2o-d+3})}\binom{\frac{d}{2}-1-o-k}{m}\binom{2k+o+1-\frac{d}{2}}{k-m}\]
	and $L_{k,m}$ is defined as above, i.e.
	\[L_{k,m+o}=w^*(G_{P-z,x})[f_s][[s^{2k+2m+2o-d+3}]][[z^{m+o}]].\]
\end{thm}
\begin{rem}
	\label{qsipm}
	There are some choices of $o$ that make the formulas a bit simpler:
	\begin{enumerate}
		\item For $o=0$, we get
			\[q(k,m,0)=\frac{\pi^{\frac{d}{2}-1}4^{k+m}m!k!}{\M'(f)({2k+2m-d+3})}\binom{\frac{d}{2}-1-k}{m}\binom{2k+1-\frac{d}{2}}{k-m}\]
		\item For $d$ even and $o=\frac{d}{2}-1$, we obtain
			\[q(k,m,\frac{d}{2}-1)=\frac{(4\pi)^{\frac{d}{2}-1}4^{k+m}(m+\frac{d}{2}-1)!k!}{\M'(f)({2k+2m+1})}\binom{-k}{m}\binom{2k}{k-m}.\]
		\item For $d$ even, $k\leq \frac{d}{2}-1$ and $o=\frac{d}{2}-1-k$, we obtain for $m\neq 0$:
		\[q(k,m,o)=0,\]
		as the first binomial coefficient vanishes.
		The only non-vanishing summand comes from
		\[q(k,0,o)=\frac{(4\pi)^{\frac{d}{2}-1}(\frac{d}{2}-1-k)!k!}{\M'(f)({1})}.\]
		This much simpler formula reflects the fact that the $\frac{d}{2}-1$ coefficient can be extracted from the expansion directly, and the offset $o$ shifts the $k$-th Hadamard coefficient to that position.
	\end{enumerate}	
\end{rem}
	The first Hadamard coefficient is related to the scalar curvature via (\cite[before 2.3.2]{BGP}):
	\[V^1_x(x)=\frac{1}{6}\scal(x)\]
	Thus we obtain as a special case of the above:
	\begin{cor}
		For $d>2$ even, we have
		\[\scal(x)=\frac{6(4\pi)^{\frac{d}{2}-1}(\frac{d}{2}-2)!}{\M'(f)({1})}w^*(G_{\square-z,x})[f_s][[s^{1}]][[z^{\frac{d}{2}-2}]].\]
		in particular, for $d=4$:
		\[\scal(x)=\frac{24\pi}{\M'(f)({1})}w^*(G_{\square,x})[f_s][[s^{1}]].\]
	\end{cor}
Instead of using Green's  operators for $P-z$, one can use powers of Green's operators instead. The reason, morally, why this works is the following calculation (viewing $P$ on past/future compactly supported distributions):
\begin{align*}
	G^\pm_{P-z}&=(P-z)^{-1}=P^{-1}(1-zP^{-1})^{-1}\\
	&=P^{-1}\sum\limits_{n=0}^\infty z^n P^{-n}\\
	&=\sum\limits_{n=0}^\infty z^n (G^\pm)^{n+1},
\end{align*}
so (at least formally)
\[G^\pm_{P-z}[[z^m]]=(G^\pm)^{m+1}.\]
Indeed, one obtains the following result (\cite[Proposition 5.2.17]{Ron}):
\begin{prop}
	we have for any $m,j\in \N$:
	\[L_{k,m}=w^*\left (G^{+\ m+1}(\delta_x)-G^{-\ m+1}(\delta_x)\right )[f_s][[s^{2k+2m+3-d}]].\]
\end{prop}
Inserting this into Theorem \ref{main} yields
\begin{cor}
	\label{Vwithpowers}
	Assume that $\M'(f)$ is non-zero at integers.
	We have
	\[V^k_x(x)=\sum\limits_{m=0}^k q(k,m,o)w^*\left (G^{+\ m+o+1}(\delta_x)-G^{-\ m+o+1}(\delta_x)\right )[f_s][[s^{2k+2m+2o+3-d}]]\]
	with $q(k,m,o)$ as in Theorem \ref{main}.
\end{cor}
\section{Hadamard coefficients from the product}
\label{ProdDiag}
In this section, we present a different approach towards extracting the Hadamard coefficients that will not give us the somewhat unwieldy sums that we obtained in the previous section. It will also no longer require our timelike curve to be a geodesic. For this end, we will take a metric product with the real line.
\begin{df}
	\label{dfMtilde}
Let
\[\tilde M:= M\times \R\]
with the product metric (with the standard metric on $\R$, i.e. the factor $\R$ is spacelike). On $\tilde M$, we consider the product operator
\[\tilde P:= P-\partial_t^2\]
on the pull-back bundle of $E$ under the product projection. We set $\tilde U:=U\times \R$.
\end{df}
The extra information we will need will now no longer come from an external parameter $z$ but from the parameter of the extra factor of the product.
In order to apply the analysis from the old setting, we want to show that the properties required of $U$ carry over to $\tilde U$. To show this, we first need the following.
\begin{lemma}
	$U$ is causally convex, i.e. if $\kappa$ is a future oriented causal curve in $M$ connecting $x,x'\in U$, then $\kappa$ must lie entirely in $U$.
\end{lemma}
\begin{proof}
	 Assume that $\kappa$ does not lie entirely in $U$. Furthermore, assume without loss of generality that $\kappa$ is defined on $[0,1]$. let \[a:=\min\{t\in[0,1]\mid \kappa(t)\notin U\}\] and \[b:=\max\{t\in[0,1]\mid \kappa(t)\notin U\}.\]
	 Extend $\kappa_{[0,a)}$ backwards and $\kappa_{(b,0]}$ forwards to obtain inextendible curves $\kappa_1$ and $\kappa_2$ in $U$. 
	Let $\Sigma$ be a cauchy hypersurface in $U$. As both $\kappa_1$ and $\kappa_2$ must intersect $\Sigma$, crossing from its past to its future, we must have $t_1\in [0,a)$ and $t_2\in(b,1]$ such that $\kappa(t_1)$ lies in the future of $\Sigma$ and $\kappa(t_2)$ lies in the past of $\Sigma$.
	 Then restricting $\kappa$ yields a future causal curve in $M$ from $\kappa(t_1)$ to $\kappa(t_2)$. By causal compatibility, there must be a future causal curve in $U$ that also joins those points. We can extend this curve by following $\kappa_1$ and $\kappa_2$ to obtain a future causal curve that intersects $\Sigma$ more than once, which is a contradiction.
\end{proof}
We can now show that $\tilde U$ satisfies everything we need.
\begin{prop}
	$\tilde U$ is again globally hyperbolic, geodesically convex and causally compatible in $\tilde M$.
\end{prop}
\begin{proof}
	An arbitrary curve in $\tilde M$ is of the form
	\[\tilde \kappa=(\kappa,\rho),\]
	where $\kappa$ is a curve in $M$ and $\rho$ is a curve in $\R$. As $\tilde M$ has product metric, the Levi-Civita connection also has product structure, so $\tilde \kappa$ is a geodesic if and only if $\kappa$ and $\rho$ are.
	Thus any two points $(x,r)$ and $(x',r')$ are joined by a unique (up to reparametrization) geodesic in $\tilde U$: the pair of the geodesic from $x$ to $x'$ in $U$ and the  geodesic from $r$ to $r'$ in $\R$. Thus $\tilde U$ is geodesically convex.

	Let $\Sigma$ be a Cauchy hypersurface in $U$ and set $\tilde \Sigma=\Sigma\times \R$. Suppose $\tilde\kappa\colon (0,1)\rightarrow \tilde U$ is an inextendible causal curve.  As $\tilde\kappa$ is causal, we must have \[-\<\kappa',\kappa'\>\geq (\rho')^2,\]
	in particular $\kappa$ must also be causal.  As $\tilde \kappa$ is inextendible, $\kappa$ or $\rho$ must be inextendible. We want to show that $\kappa$ must be inextendible either way.
	Suppose that $\rho$ is inextendible. Then it must have infinite length. We get
	\[\int\limits_0^1 \sqrt{|\<\kappa'(t),\kappa'(t)\>|}dt\geq \int\limits_0^1 |\rho'(t)|dt=\infty.\]
	Thus $\kappa$ must have infinite negative length, so it is inextendible either way. As an inextendible causal curve, $\kappa$ must intersect $\Sigma$ exactly once. Thus $\tilde \kappa$ must intersect $\tilde \Sigma$ exactly once. We conclude that $\tilde \Sigma$ is a cauchy hypersurface for $\tilde U$, so $\tilde U$ is globally hyperbolic.

	Finally, suppose that $(x,r)$ and $(x',r')$ are connected by a future oriented causal curve $\tilde k=(\kappa,\rho)$ in $\tilde M$. Then $\kappa$ is a causal curve from $x$ to $x'$. As $U$ is causally convex, $\kappa$ lies in $U$. Thus $\tilde \kappa$ lies in $\tilde U$.
\end{proof}
We see that $\tilde M$, $\tilde P$ and $\tilde U$ satisfy the conditions we placed on $M$, $P$ and $U$. Thus all constructions made with the latter also work with the former.
\begin{df}
	\label{dftilde}
	For quantities constructed from $M$, $P$ and $U$, denote with a tilde the counterpart constructed from $\tilde M$, $\tilde P$ and $\tilde U$. E.g. $\tilde V^k$ denotes the Hadamard coefficients of $\tilde P$ on $\tilde U$ and $\tilde R$ denotes the Riesz distributions on $\tilde U$. For $x\in M$ write $\tilde x$ for $(x,0)$.
\end{df}

The crucial fact that we will exploit is that the new Hadamard coefficients are essentially the old ones. More precisely, we have for any $x,y\in M$, and $r,s\in \R$
\[\tilde V^k_{(x,r)}(y,s)=V^k_x(y),\]
as can be checked by inserting this into the transport equations for $\tilde P$, which differ from those of $P$ only by time derivatives.
This identity will allow us to extract the Hadamard coefficients from the Green's operators on $\tilde M$. As the Hadamard coefficients do not change under a shift in $\R$ direction, while the Riesz distributions do, we can now detect the difference between these two contributions and discard the error terms coming from lower order Riesz-distributions. Looking again at our model case of a power series, our problem becomes analogous to extracting the coefficients $a_k(0)$ from the series
\[\insum k a_k(x)(x^2-r^2)^\frac{k}{2}\]
which is solvable due to the additional $r$-dependence.
\begin{df}
	Let again $w$ be a future oriented timelike curve in $M$ with $w(0)=x$ and initial velocity $\nu_w (0)=1$. For $\xi>1$, define
	\[w_\xi(t):=(w(\xi t), \sqrt{\xi^2-1}\ t).\]
\end{df}
To deal with the derivative terms that will come up after applying Theorem \ref{intexp}, we will need the following Lemma:
\begin{lemma}
	\label{xidiff}
	With
	\[\nu_{w_\xi}(t):=\frac{\tilde \Gamma_{\tilde x}(w_\xi(t))}{t^2},\]
	we have for any $p\in \R$ that
	\[(\nu_{w_\xi}^p \tilde V^k_x\circ w_\xi)(0)=V^k_x(x)\]
	and for $n\geq 1$, its  $n$-th derivative
	\[(\nu_{w_\xi}^p \tilde V^k_x\circ w_\xi)^{(n)}(0)\]
	is a polynomial in $\xi$ without any summands of order less than $n$.
\end{lemma}
\begin{proof}
	We have
	\[\nu_{w_\xi}(t)=\frac{1}{t^2}(\Gamma_x(w(\xi t))-(\xi^2-1)t^2)=\xi^2 \frac{\Gamma_x(w(\xi t))}{(\xi t)^2}-\xi^2+1=\xi^2(\nu_w(\xi t)-1)+1\]
	for $t=0$, the first summand vanishes, so
	\[\nu_{w_\xi}(0)=1.\]
	Together with
	\[\tilde V^k_x(w_\xi(t))=V^k_x(w(\xi t)),\]
	this implies the first claim.
	
	For this proof, define a "good summand of order $k$" to be a function of the form
	\[g_\xi(t)=\xi^k \nu_{w_\xi}(t)^qh(\xi t)\]
	for some $q\in \R$ and some smooth function $h$. The derivative of this is
	\[g_\xi'(t)=q\xi^{k+3} \nu_{w_\xi}(t)^{q-1}\nu_w '(\xi t)h(\xi t)+\xi^{k+1} \nu_{w_\xi}(t)^qh'(\xi t),\]
	which is a sum of good summands of order at least $k+1$. As $(\nu_{w_\xi}^p V^k_x\circ w_\xi)$ is a good summand of order zero, we can conclude by induction that $(\nu_{w_\xi}^p \tilde V^k_x\circ w_\xi)^{(n)}$ is a sum of good summands of order at least $n$. Evaluating a good summand at $0$ yields
	\[g_\xi(0)=\xi^k h(0),\]
	i.e. a monomial in $\xi$ of the same order as the summand. This proves the second claim.
\end{proof}
We now have everything in hand to show the main theorem of this section:
\begin{thm}
	\label{prodmain}
	Assume that $\M '(f)(2k+2-d)\neq 0$. We have
	\[V^k_x(x)=\frac{4^k k!}{\pi^{\frac{1-d}{2}}\M'(f)(2k+2-d)}w_\xi^*(\tilde G)[f_s][[s^{2k+2-d}]][[\xi^0]].\]
\end{thm}
\begin{proof}
	Applying Theorem \ref{intexp} for $w_\xi$ as $w$ and $\tilde P$ as $P$ yields
	\[w_\xi^*(\tilde G)[f_s]\sim\insum{k,n} a(k,n,f,d+1)(\nu_{w_\xi}^{k-\frac{d-1}{2}} \tilde V^k_x\circ w_\xi)^{(2n)}(0)s^{2k+2n+2-d}.\]
	Thus we have
	\[w_\xi^*(\tilde G)[f_s][[s^{2k+2-d}]]=\sum\limits_{l+n=k}a(l,n,f,d+1)(\nu_{w_\xi}^{l-\frac{d-1}{2}} \tilde V^l_x\circ w_\xi)^{(2n)}(0).\]
	Applying Lemma \ref{xidiff} tells us that we only get a constant part for $n=0$, which is
	\begin{align*}
		w_\xi^*(\tilde G)[f_s][[s^{2k+2-d}]][[\xi^0]]&=a(k,0,f,d+1)V^k_x(x)
	\end{align*}
	Rearranging, inserting the definition of $a$ and using that $\binom{\alpha}{0}$ is always $1$, we obtain
	\begin{align*}
		V^k_x(x)&=a(k,0,f,d+1)^{-1}w_\xi^*(\tilde G)[f_s][[s^{2k+2-d}]][[\xi^0]]] \\
		&=\frac{4^k k!}{\pi^{\frac{1-d}{2}}\M'(f)(2k+2-d)}w_\xi^*(\tilde G)[f_s][[s^{2k+2-d}]][[\xi^0]]].
	\end{align*}
\end{proof}
\section{Off-diagonal coefficients}
\label{ProdOffDiag}
The product approach can also be used to obtain values of the Hadamard coefficients away from the diagonal. Here we use a geodesic in the new space direction passing through a point $\tilde y$ rather that a timelike geodesic passing through $\tilde x$. Instead of the behavior at $\tilde x$, we now use the different vanishing behavior of the various Riesz distributions at the intersection of a space-line with the lightcone. Instead of a cut-off around zero, we need to rescale a cut-off function around these intersection points.
\begin{df}
	For $y\in I^\pm(x)$, where $I^\pm$ denotes the timelike future/past, define \[\iota_y(r):=(y,r).\] Fix $\chi\in C_c^\infty((-1,1))$ and set \[\chi_{\epsilon,x,y}(r):=|r|\chi\left(\frac{\Gamma_x(y)-r^2}{\epsilon}\right).\]
\end{df}
Once more we use wavefront calculus to show that our distributions can be restricted to the curves we want to consider. 
\begin{prop}
	 Let $y\in I_\pm(x)$. Then \penalty-50\hskip0pt $\iota_y^*(\tilde R^\pm(\alpha,\tilde x))$ is  well-defined and holomorphic and $\iota_y^*(\tilde G^\pm_{\tilde x})$ is well-defined.
\end{prop}
\begin{proof}

	Using Theorem \ref{WFRiesz} and the fact that wavefront calculus is compatible with diffeomorphisms (in this case the exponential map), we obtain that the Riesz distributions $\tilde R(\alpha,\tilde x)$ are holomorphic (as a function of $\alpha$) in the space of distributions whose wavefront away from $\tilde x$ consists of multiples of $d\Gamma_{\tilde x}$ at points where $\Gamma_{\tilde x}=0$ (in particular, $r\neq 0$). At (preimages of) such points
	\[d(\tilde\Gamma_{\tilde x}\circ\iota_y)(r)=2r\]
	is non-zero, so $d\tilde \Gamma_{\tilde x}$ is not conormal to the image of $\iota_y$. Thus $\tilde R(\alpha,{\tilde x})$ is holomorphic in the space of distributions whose wavefront set is not conormal to $\iota_y$, so the pullback is well-defined and holomorphic. By using the Hadamard expansion with continuous remainder, this implies that the pullback of $\tilde G^\pm_{\tilde x}$ is also well-defined.
\end{proof}
Now we can perform the restriction and pair with the cut-off we chose. We first compute the result for a single summand in the Hadamard expansion.
\begin{lemma}
	\label{summandev}
 For $k\in\N$, $\alpha\in \C$, $y\in I^\pm(x)$ and $\epsilon<\Gamma_x(y)$, we have 
 \[\iota_y^*(\tilde V^k_{\tilde x}\tilde R^\pm({\alpha},{\tilde x}))[\chi_{\epsilon,x,y}]=\frac{2^{1-\alpha}\pi^{\tfrac{1-d}{2}}}{\Gamma(\frac{\alpha}{2})}\frac{\M(\chi)}{\Gamma}(\tfrac{\alpha+1-d}{2})V^k_x(y)\epsilon^{\tfrac{\alpha+1-d}{2}}\]
\end{lemma}
\begin{proof}
	For $\Re(\alpha)>0$ large enough (and $\alpha$ not a zero of $\tilde c_\alpha$), we have
	\begin{align*}
		&\iota_y^*(\tilde V^k_{\tilde x}\tilde R({\alpha},{\tilde x}))[\chi_{\epsilon,x,y}]\\
		&=\intR \tilde V^k_{\tilde x}(y,r)R({\alpha},{\tilde x})(y,r)\chi_{\epsilon,x,y}(r) dr\\
		&=\int\limits_{r^2<\Gamma_x(y)} V^k_{x}(y)\tilde c_\alpha(\Gamma_x(y)-r^2)^{\tfrac{\alpha-d-1}{2}}|r|\chi\left(\frac{\Gamma_x(y)-r^2}{\epsilon}\right) dr\\
		&=\int\limits_0^{\Gamma_x(y)} V^k_{x}(y)\tilde c_\alpha s^{\tfrac{\alpha-d-1}{2}}\chi\left(\frac{s}{\epsilon}\right) ds\\
		&=\epsilon \int\limits_0^{\tfrac{\Gamma_x(y)}{\epsilon}} V^k_{x}(y)\tilde c_\alpha (\epsilon t)^{\tfrac{\alpha-d-1}{2}}\chi\left(t\right) dt\\
		&=\epsilon^{\tfrac{\alpha-d+1}{2}} \tilde c_\alpha V^k_{x}(y)\int\limits_0^{\infty}   t^{\tfrac{\alpha-d-1}{2}}\chi\left(t\right) dt\\
		&=\frac{2^{1-\alpha}\pi^{\frac{1-d}{2}}}{\Gamma(\tfrac{\alpha}{2})\Gamma(\tfrac{\alpha-d+1}{2})} \M(\chi)(\tfrac{\alpha-d+1}{2})V^k_{x}(y)\epsilon^{\tfrac{\alpha-d+1}{2}}.\\
	\end{align*}
	As both sides are holomorphic in $\alpha$, equality holds for all $\alpha$.
\end{proof}
The Hadamard expansion now leads to an asymptotic expansion for the expression with the full Green's operators.
\begin{thm}
	For $y\in I^\pm(x)$ we have as an asymptotic expansion in $\epsilon\to0$:
	\[\iota_y^*(\tilde G^\pm_{\tilde x})[\chi_{\epsilon,x,y}]\sim \insum{k}\frac{\pi^{\tfrac{1-d}{2}}}{2^{2k+1}k!}\frac{\M(\chi)}{\Gamma}(k+\tfrac{3-d}{2})V^k_x(y)\epsilon^{k+\tfrac{3-d}{2}}\]
\end{thm}
\begin{proof}
	We use the Hadamard expansion
	\[\tilde G^\pm_{\tilde x}=\sum\limits_{k=0}^N\tilde V^k_{\tilde x}\tilde R^\pm(\alpha,{\tilde x})+E_N\]
	where for any $m\in \N$, we can get $E_N\in C^m$ by choosing $N$ large enough. We have already taken care of the summands, so it remains to estimate the remainder term. Note that as $G^\pm$ and $R^\pm$ are supported in $J^\pm(x)$, so is $E_N$. This means that it vanishes of order $m$ near the boundary of the lightcone. As $\tilde\Gamma_{\tilde x}$ vanishes of order one at $(y,\pm\sqrt{\Gamma_x(y)})$, we have
	\[|E_N(y,r)|\leq C|\tilde \Gamma_{\tilde x}|^m\]
	for some constant $C$.
	Since $\chi_{\epsilon,x,y}(r)$ vanishes if $|\tilde\Gamma_{\tilde x}(r)|>\epsilon$, we obtain
	\begin{align*}
		|\iota_y^*(E_N)[\chi_{\epsilon,x,y}]|&\leq\intR |E_N(y,r) \chi_{\epsilon,x,y}(r)| dr\\
		&\leq C\epsilon^m\intR|\chi_{\epsilon,x,y}(r)| dr\\
		&=O(\epsilon^m).
	\end{align*}
	Using Lemma \ref{summandev} (with $\alpha=2k+2$), we can now conclude
	\begin{align*}
		&\iota_y^*(\tilde G^\pm_{\tilde x})[\chi_{\epsilon,x,y}]\\
		&=\sum\limits_{k=0}^N\iota_y^*(\tilde V^k_{\tilde x}\tilde R^\pm(\alpha,{\tilde x}))[\chi_{\epsilon,x,y}]+\iota_y^*(E_N)[\chi_{\epsilon,x,y}]\\
		&=\sum\limits_{k=0}^N\frac{2^{-1-2k}\pi^{\tfrac{1-d}{2}}}{\Gamma(\frac{2k+2}{2})}\frac{\M(\chi)}{\Gamma}(\tfrac{2k+3-d}{2})V^k_x(y)\epsilon^{\tfrac{2k+3-d}{2}}+O(\epsilon^m)\\
		&=\sum\limits_{k=0}^N\frac{\pi^{\tfrac{1-d}{2}}}{2^{2k+1}k!}\frac{\M(\chi)}{\Gamma}(k+\tfrac{3-d}{2})V^k_x(y)\epsilon^{k+\tfrac{3-d}{2}}+O(\epsilon^m)
	\end{align*}
	As $m$ can be chosen arbitrarily by choosing $N$ large enough, this concludes the proof.
\end{proof}
\begin{rem}
	Looking at the remainder term in the proof, we see that it is locally uniformly bounded. By making $N$ bigger, we can ensure the same for any fixed derivative of the remainder term. Thus if we want, we can see the above asymptotic expansion in $C^\infty(U\times U)$ rather than pointwise.
\end{rem}

Taking the coefficient for a fixed power of $\epsilon$ and solving for $V^k_x(y)$, one immediately obtains the following.
\begin{cor}
\label{offdV}
	Assume that $\frac{\M(\chi)}{\Gamma}$ does not vanish at $k+\tfrac{3-d}{2}$.
	Then for $y\in I^\pm(x)$ and $k\in \N$, we have
	\[V^k_x(y)=2^{2k+1}k!\pi^{\tfrac{d-1}{2}}\frac{\Gamma}{\M(\chi)}(k+\tfrac{3-d}{2})\iota_y^*(\tilde G^\pm_x)[\chi_{\epsilon,x,y}][[\epsilon^{k+\tfrac{3-d}{2}}]].\]
\end{cor}
For a global version of this, we can define
\[\chi_\epsilon(x,r',y,r):=\chi_{\epsilon,x,y}(r-r')\]
and
\[\pi(x,r,y,r'):=(x,r',y).\]
Then pushforward by $\pi$ corresponds to integration over the last variable. Wavefront calculus tells us that $\pi_*(\tilde G^\pm)$ is well-defined and smooth and that
\[\iota_y^*(\tilde G^\pm)[\chi_{\epsilon,x,y}]=\pi_*(\chi_\epsilon\tilde G^\pm)(x,0,y).\]
Thus the above can be rephrased as
\[V^k=2^{2k+1}k!\pi^{\tfrac{d-1}{2}}\frac{\Gamma}{\M(\chi)}(k+\tfrac{3-d}{2})\pi_*(\chi_\epsilon\tilde G^\pm)(\cdot,0,\cdot)[[\epsilon^{k+\tfrac{3-d}{2}}]].\]

\printbibliography
\end{document}